\documentclass{amsart}
\usepackage{amssymb,stmaryrd}

\newtheorem{thm}{Theorem}[section]
\newtheorem*{thma}{Theorem~A}
\newtheorem*{thmb}{Theorem~B}

\newtheorem{cor}[thm]{Corollary}
\newtheorem{claim}{Claim}[thm]
\newtheorem{prop}[thm]{Proposition}

\theoremstyle{definition}
\newtheorem{defn}[thm]{Definition}

\theoremstyle{remark}
\newtheorem{remark}[thm]{Remark}

\DeclareMathOperator{\U}{U}
\DeclareMathOperator{\MA}{\sf MA}
\DeclareMathOperator{\gma}{\sf GMA}
\DeclareMathOperator{\zfc}{\sf ZFC}
\DeclareMathOperator{\gch}{\sf GCH}
\DeclareMathOperator{\im}{Im}
\DeclareMathOperator{\cf}{cf}
\DeclareMathOperator{\otp}{otp}
\DeclareMathOperator{\dom}{dom}
\DeclareMathOperator{\add}{Add}
\newcommand{\s}{\subseteq}
\renewcommand\mid{\mathrel{|}\allowbreak}

\subjclass[2010]{Primary 03E02; Secondary 03E35}

\title[Ramsey theory over partitions I]{Ramsey theory over partitions I:\\Positive Ramsey relations from forcing axioms}

\author[M. Kojman]{Menachem Kojman}
\address{Department of Mathematics, Ben-Gurion University of the Negev, P.O.B. 653, Be’er Sheva, 84105 Israel}
\urladdr{https://www.math.bgu.ac.il/~kojman/}

\author[A. Rinot]{Assaf Rinot}
\address{Department of Mathematics, Bar-Ilan University, Ramat-Gan 5290002, Israel.}
\urladdr{http://www.assafrinot.com}

\author[J. Stepr\={a}ns]{Juris Stepr\={a}ns}
\address{Department of Mathematics \& Statistics, York University, 4700 Keele Street, Toronto, Ontario, Canada M3J 1P3}
\urladdr{http://www.math.yorku.ca/~steprans/}

\date{Preprint as of April 16, 2022. For the latest version, visit \textsf{http://assafrinot.com/paper/49}.}

\begin{document}
\begin{abstract}   In this series of papers, we advance Ramsey theory of colorings over partitions.
In this part, a correspondence between anti-Ramsey properties
of partitions and chain conditions of the natural forcing notions  that
homogenize  colorings over them is uncovered. 
At the level of the first uncountable cardinal this gives rise to a duality theorem under Martin's Axiom: a function $p:[\omega_{1}]^{2}\rightarrow\omega$ witnesses a weak negative Ramsey relation when $p$ plays the role of  a coloring if and only if a positive Ramsey relation holds over $p$ when $p$ plays the role of a partition.

The consistency of positive Ramsey relations over partitions does not stop at the first uncountable cardinal:
it is established that at any prescribed uncountable cardinal these relations follow
from forcing axioms without large cardinal strength. This result solves in
  particular two problems from \cite{strongcoloringpaper}.
\end{abstract}
\maketitle

\section{Introduction}

\subsection{Ramsey relations} For (finite or infinite) cardinals $\kappa$, $\lambda$ and $\theta$, the Ramsey relation
$\kappa\rightarrow(\lambda)^2_\theta$ asserts that for every coloring $c:[\kappa]^2\rightarrow\theta$ of unordered pairs of
ordinals in $\kappa$ by $\theta$ colors there is a $c$-homogeneous set of size $\lambda$, i.e.,
there exist $A\s\kappa$ of size $\lambda$ and $\tau\in\theta$ such that $c(\alpha,\beta)=\tau$ for any pair $\alpha<\beta$
of elements of $A$. In this  notation, Ramsey's famous theorem
\cite{ramsey} is written as $\aleph_0\rightarrow(\aleph_0)^2_2$.

If $\aleph_{0}$  is
substituted in this relation by the next infinite cardinal, $\aleph_{1}$, then the relation fails,
meaning that its negation 
$\aleph_1\nrightarrow(\aleph_1)^2_2$ holds. In fact, Sierpi{\'n}ski \cite{MR1556708} proved that
$2^\lambda\nrightarrow(\lambda^+)^2_2$ holds for every infinite cardinal $\lambda$ and as $\lambda^{+}\le 2^{\lambda}$, by
monotonicity, $\kappa\nrightarrow(\kappa)^2_2$  for every successor cardinal $\kappa=\lambda^+$.

Later on, work by Tarski, Erd\H os, Hanf and others (see \cite[Chapter~2]{Kanamori}) has
shown that if the relation $\kappa\rightarrow(\kappa)^{2}_{2}$ happens to  hold for some cardinal $\kappa>\aleph_0$
then $\kappa$ must be a weakly compact cardinal, a large cardinal which is
inaccessible and satisfies the higher analog of K\"onig's lemma, that any tree of size $\kappa$ all of whose levels have
size $<\kappa$ admits a branch of size $\kappa$.
As the existence of an inaccessible cardinal cannot be proved from the axiomatic system   of
set theory, $\zfc$,\footnote{Unless $\zfc$ is inconsistent.} this established that in $\zfc$ alone it would be  impossible to
prove that any cardinal $\kappa$ other than
$\aleph_{0}$ satisfies the positive Ramsey relation $\kappa\rightarrow(\kappa)^{2}_{2}$.

A weaker Ramsey relation than the one above is the square brackets relation
$\kappa\rightarrow[\lambda]^2_\theta$, which asserts that for every coloring $c:[\kappa]^2\rightarrow\theta$ there exist
$A\s\kappa$ of size $\lambda$ and $\tau\in\theta$ such that $c(\alpha,\beta)\neq\tau$ for any pair $\alpha<\beta$ of elements
of $A$. Erd\H os, Hajnal and Rado \cite{MR202613} proved $\lambda^+\nrightarrow[\lambda^+]^2_{\lambda^+}$
(i.e. the negation of  $\lambda^{+}\rightarrow[\lambda^{+}]^{2}_{\lambda^{+}}$) for every infinite cardinal $\lambda$ at
which the Generalized Continuum Hypothesis ($\gch$) holds, that is, for every  $\lambda$ for which
$2^\lambda=\lambda^+$.

Then, in \cite{TodActa}, Todor{\v{c}}evi{\'c} waived the hypothesis $2^\lambda=\lambda^+$
for infinite cardinals $\lambda$ that are regular, establishing in particular that the
negative Ramsey relation $\aleph_1\nrightarrow[\aleph_1]^2_{\aleph_1}$ is a theorem of $\zfc$. Thus, the
analog of Ramsey's theorem for $\aleph_{0}$ fails strongly on all successors of regular
cardinals outright in $\zfc$: on every such $\lambda^{+}$ there is a coloring of pairs by
$\lambda^{+}$ colors with the property that no color is omitted on the pairs from any
subset $A\subseteq \lambda^{+}$ of full size.

In the
positive direction, Shelah \cite{MR955139} showed that Sierpi{\'n}ski's theorem
is optimal in the sense that the consistency of a measurable cardinal 
is sufficient to get the consistency of $2^{\aleph_0}\rightarrow[\aleph_1]^2_3$. More works in
the positive direction at the level of the first uncountable cardinal $\aleph_1$
include \cite{MR716846,MR4190059}, and the fact that the forcing axiom $\MA$
(Martin's Axiom) implies that the product of any two $ccc$ topological spaces is
again $ccc$. The point is that the Ramsey relation $\kappa\rightarrow(\kappa)^2_2$ naturally reduces
the question of $\kappa$-$cc$ of the product of two spaces to the question of $\kappa$-$cc$ of each factor,
so $\MA$ yields a
consequence of $\aleph_1\rightarrow(\aleph_1)^2_2$, even though the relation itself fails in $\zfc$.

However, this echo of a positive Ramsey relation on $\aleph_{1}$ turned out to be a very  isolated case:
from the \emph{second} uncountable cardinal $\aleph_2$ on, for every successor
cardinal $\kappa$, including successors of singular cardinals, there is a $\kappa$-cc
space whose square violates the $\kappa$-cc (see \cite{paper18} for a comprehensive
account of productivity of chain conditions and colorings).

In the positive direction, Mitchell
and Silver \cite{MR313057} proved that the consistency of a weakly compact
cardinal is necessary and sufficient for the higher analog of K\"onig's lemma to
hold at the level of $\aleph_2$.

\medskip In summary, the higher one goes on the scale of the Alephs, Ramsey
theory steers away from Ramsey's original theorem towards strong negative
relations. The consistency of positive relations on uncountable cardinals is rare
and requires large cardinal strength.

\subsection{Ramsey relation over partitions} Recently, a new, weaker type of Ramsey relations \emph{over partitions} was discovered \cite{strongcoloringpaper}.
Given a partition $p:[\kappa]^{2}\rightarrow\mu$ of the unordered pairs from $\kappa$ into $\mu$ cells,
it is possible  to relax  the notion of homogeneity to \emph{relative} homogeneity over $p$. Declare a set
$A\subseteq \kappa$ as \emph{homogeneous over $p$} for a coloring $c:[\kappa]^{2}\rightarrow\theta$, or
\emph{$(p,c)$-homogeneous} for short, if all pairs from $A$ which lie in any
single $p$-cell are colored by one color which depends on the cell. More formally, for every $j<\mu$,
the set $\{c(\alpha,\beta)\mid(\alpha,\beta)\in [A]^{2}\ \&\ p(\alpha,\beta)=j\}$ has size no more than $1$.

The
standard positive Ramsey relation $\kappa\rightarrow(\lambda)^{2}_{\theta}$ can now be relaxed, for a
partition $p$, to its ``over $p$'' version, $\kappa\rightarrow_{p}(\lambda)^{2}_{\theta}$, to mean that for
every coloring $c:[\kappa]^{2}\rightarrow\theta$ there is a set $A\subseteq \kappa$ of size $\lambda$ which is
\emph{homogeneous over $p$}. Clearly, if $A$ is homogeneous for $c$, it is also
$(p,c)$-homogeneous for every partition $p$ of the pairs from $\kappa$, so
$\kappa\rightarrow_{p} (\lambda)^{2}_{\theta}$ follows
from $\kappa\rightarrow(\lambda)^{2}_{\theta}$ for any partition $p$, but it is feasible that for some $p$, $\kappa\rightarrow_{p} (\lambda)^{2}_{\theta}$, even in situations where $\kappa\nrightarrow (\lambda)^{2}_{\theta}$.

When $A\subseteq\kappa$ is
homogeneous over $p$, the $p$-cell $p(\alpha,\beta)$ of a pair from $A$ determines its
color $c(\alpha,\beta)$, so there is a function $\tau:\mu\rightarrow\theta$ such that $c(\alpha,\beta)=\tau(p(\alpha,\beta))$ for
every $(\alpha,\beta)\in [A]^{2}$.

Similarly, when putting $p$ in the weaker square brackets Ramsey relation
$\kappa\rightarrow[\lambda]^{2}_{\theta}$, the relation $\kappa\rightarrow_{p} [\lambda]^{2}_{\theta}$ means that for every
coloring $c:[\kappa]^{2}\rightarrow\theta$ there is a set $A\subseteq \kappa$ of cardinality $\lambda$ such that at
least one color from $\theta$ is \emph{omitted} by $c$ in every $p$-cell intersected
with $[A]^{2}$. In this case there will be a function $\tau:\mu\rightarrow\theta$
such that
$c(\alpha,\beta)\not=\tau(p(\alpha,\beta))$ for every $(\alpha,\beta)\in [A]^{2}$.

In \cite{strongcoloringpaper} it was indeed shown that
$\aleph_{1}\rightarrow_{p}[\aleph_{1}]^{2}_{\aleph_{1}}$ is consistent in a $ccc$ forcing extension for some generic  partition
$p:[\aleph_{1}]^{2}\rightarrow \aleph_{0}$. In other words, although $\aleph_1\nrightarrow[\aleph_{1}]^{2}_{\aleph_{1}}$
holds absolutely, by Todor{\v{c}}evi{\'c}'s theorem, it is consistent that for some
countable partition $p$ the opposite, positive Ramsey relation
$\aleph_{1}\rightarrow_p[\aleph_{1}]^{2}_{{\aleph_{1}}}$ holds.

This consistency result led to many interesting questions. Primarily, whether
$\MA$ decides Ramsey theory over countable partitions on
$\aleph_1$, and, of course, which way it decides it if it does. These problems were
stated in \cite{strongcoloringpaper} and are solved as a corollary of the results in this
paper.

\subsection{The results}  First, combinatorial properties of partitions are
listed and it is shown that: (a) partitions which satisfy those properties
exist in every model of $\zfc$; (b) $\MA$ implies that positive Ramsey relations hold over every partition
with these properties.
While the consistency results in \cite{strongcoloringpaper} rely on partitions
that are generic, in the forcing sense, the partitions considered here will be
definable from an $\omega_1$ sequence of reals.

The surprising point, though, is that this consistencey result on $\aleph_{1}$ is actually a
special case of a more general theorem which applies to every successor of a
regular cardinal via the Generalized Martin's Axiom ($\gma$). In other words,
the consistency of positive Ramsey relations over partitions does not stop
at $\aleph_1$, as is the case with classical positive Ramsey relation.

But even more is true. It is consistent that not only does every strong coloring
omit at least one color in every $p$-cell of some set $X\s \lambda^+$ of full size,
but actually the positive rounded brackets Ramsey relation $\lambda^{+}\rightarrow_{p} (\lambda^{+})^{2}_{\lambda}$ holds over a suitable partition $p:[\lambda^+]^2\rightarrow\lambda$ for an even stronger reason:
 $\lambda^+$ is the union of $\lambda$ many $(p,c)$-homogeneous sets for every coloring $c:[\lambda^+]^2\rightarrow\lambda$.

\begin{thma} 
\begin{enumerate}
\item If $\gma_{\lambda^+}$ holds for a cardinal
  $\lambda=\lambda^{<\lambda}$ then there is a partition
  $p:[\lambda^+]^2\rightarrow\lambda$ such that for
  every coloring $c:[\lambda^+]^2\rightarrow\lambda$, there is a
  decomposition $\lambda^+=\biguplus_{i<\lambda}X_i$ such that for all
  $i,j<\lambda$,
  $$|\{ c(\alpha,\beta)\mid (\alpha,\beta)\in[ X_i]^2\ \&\   p(\alpha,\beta)=j\}|=1.$$
\item If $\MA_{\aleph_1}(K)$ holds then for every partition
  $p:[\aleph_1]^2\rightarrow\aleph_0$ with finite-to-one fibers, 
  for every coloring $c:[\aleph_1]^2\rightarrow\aleph_0$ there is
  decomposition $\aleph_1=\biguplus_{i<\omega}X_i$ such that for all
  $i,j<\omega$, $$\{ c(\alpha,\beta)\mid (\alpha,\beta)\in[ X_i]^2\ \&\   p(\alpha,\beta)=j\}\text{ is finite}.$$
\end{enumerate}
\end{thma}
The main new phenomenon which is revealed here about Ramsey theory over
partitions is, then, that small partitions can consistently have positive Ramsey
relations over them at uncountable cardinals at the presence of $\gma$. Since
the consistency of $\gma$ does not require consistency assumptions beyond that
of $\zfc$, this stands in strong contrast to other characterizations of the
Ramsey properties \cite{MR313057}.

Another interesting discovery is of a duality phenomenon in
the presence of a mild forcing axiom (see Definition~\ref{mak} below): a positive Ramsey relation over $p$, when viewed as a partition, is
equivalent to $p$ witnessing a certain negative Ramsey relation when viewed as a
coloring:

\begin{thmb} 
Suppose $\MA_{\aleph_1}(K)$ holds. Then for every function $p:[\aleph_1]^2\rightarrow\aleph_0$, the following are equivalent:
\begin{enumerate}
\item $\aleph_1\rightarrow_p[\aleph_1]^2_{\aleph_0,\text{finite}}$;
\item There exists $X\in[\aleph_1]^{\aleph_1}$ such that $p\restriction[X]^2$ witnesses $\U(\aleph_1,\allowbreak\aleph_1,\aleph_0,\aleph_0)$.
\end{enumerate}
\end{thmb}

Here $\aleph_1\rightarrow_p[\aleph_1]^2_{\aleph_0,\text{finite}}$ asserts that for every coloring
$c:[\aleph_1]^2\rightarrow\aleph_0$ there exist an uncountable set $X\subseteq \aleph_{1}$ such that  all $j<\omega$ the set
$\{ c(\alpha,\beta)\mid (\alpha,\beta)\in[ X]^2\ \&\ p(\alpha,\beta)=j\}$ is finite, or, almost all colors are
omitted when restricting $c$ to the pairs from $X$ in  any single $p$-cell.
$\U(\aleph_1,\allowbreak\aleph_1,\aleph_0,\aleph_0)$ is a provable instance of the $4$-parameter
anti-Ramsey coloring principle $\U(\ldots)$ due to Lambie-Hanson and Rinot \cite{paper34}.
Its definition is reproduced in Section~\ref{prel}.

\section{Preliminaries}\label{prel}
Throughout the paper, $\kappa$ denotes a regular uncountable cardinal,
$\chi,\theta,\mu$ denote cardinals $\le\kappa$,
and $\lambda$ denotes an infinite cardinal $<\kappa$.
For sets of ordinals $a$ and $b$, we write $a < b$ if $\alpha < \beta$
for all $\alpha \in a$ and  $\beta \in b$.  
For a set $\mathcal{A}$ which is either an ordinal or a collection of sets
of ordinals, we interpret $[\mathcal{A}]^2$ as $\{ (a,b)\in\mathcal A\times\mathcal A\mid a<b\}$.
 This is a convenient means to be able to write $c(\alpha,\beta)$ instead of $c(\{\alpha,\beta\})$.
For an ordinal $\sigma>2$ and a set of ordinals $A$, we write 
$[A]^\sigma$ for $\{ B\s A\mid \otp(B)=\sigma\}$.
For a cardinal
$\chi$ and a set $\mathcal A$, we write  $[\mathcal  A]^{<\chi}:=\{\mathcal B\s\mathcal A\mid |\mathcal B|<\chi\}$.  

\begin{defn}[\cite{paper34}]\label{uprinciple} $\U(\kappa,\kappa,\mu,\chi)$ asserts the existence of a function $p:[\kappa]^2\rightarrow\mu$ such that  
for every $\sigma<\chi$, every pairwise disjoint family $\mathcal A\s[\kappa]^{\sigma}$ of size $\kappa$,
for every $\delta<\mu$, there exists $\mathcal B\s \mathcal A$ of size $\kappa$ such that $\min(p[a\times b])>\delta$ for all $(a,b)\in[\mathcal B]^2$.
\end{defn}
By \cite[Corollary~4.12]{paper34}, for every pair $\mu\le\lambda$ of infinite regular cardinals, $\U(\lambda^+,\lambda^+,\allowbreak\mu,\lambda)$ holds.

\begin{defn} Let $p:[\kappa]^2\rightarrow\mu$ be a partition. Then:
\begin{itemize}
\item $p$ has \emph{injective fibers} iff for all $\alpha<\alpha'<\beta$, $p(\alpha,\beta)\neq p(\alpha',\beta)$;
\item $p$ has \emph{finite-to-one fibers} iff for all $\beta<\kappa$ and $j<\mu$,
$\{ \alpha<\beta\mid p(\alpha,\beta)=j\}$ is finite;
\item $p$ has \emph{$\lambda$-almost-disjoint fibers}  iff  for all $\beta<\beta'<\kappa$:
$$|\{p(\alpha,\beta)\mid \alpha<\beta\}\cap\{ p(\alpha,\beta')\mid \alpha<\beta\}|<\lambda;$$
\item $p$ has \emph{$\lambda$-coherent fibers} iff  for all $\beta<\beta'<\kappa$:
$$|\{\alpha<\beta\mid p(\alpha,\beta)\neq p(\alpha,\beta')\}|<\lambda;$$
\item $p$ has \emph{$\lambda$-Cohen fibers} 
iff for every injection $g:a\rightarrow\mu$ with $a\in[\kappa]^{<\lambda}$,
there are cofinally many $\beta<\kappa$ such that $g(\alpha)=p(\alpha,\beta)$ for all $\alpha\in a$.
\end{itemize}
\end{defn}

\begin{prop}\label{Wildkjsbdkj} Suppose that $\lambda$ is an infinite regular cardinal.
\begin{enumerate}
\item There is a partition $p:[\lambda^+]^2\rightarrow\lambda$
with injective and $\lambda$-coherent fibers;
\item For every cardinal $\chi$ such that $\lambda^{<\chi}=\lambda$, there is a partition $p:[\lambda^+]^2\rightarrow\lambda$
with injective, $\lambda$-almost-disjoint and $\chi$-Cohen fibers.
\end{enumerate}
\end{prop}
\begin{proof}  
(1) See, for instance, \cite[Lemma~6.25]{TodWalks}.

(2) Assuming $\lambda^{<\chi}=\lambda$, fix an enumeration $\langle g_\beta\mid \beta<\lambda^+\rangle$
of all injections $g$ with $\dom(g)\in [\lambda^+]^{<\chi}$ and $\im(g)\s \lambda$
in which each such injection occurs cofinally often.
For each $\beta<\lambda^+$, let $\gamma_\beta:=\sup(\im(g_\beta))+1$.

Fix a sequence $\vec Z=\langle Z_\beta\mid \beta<\lambda^+\rangle$ of elements of $[\lambda]^\lambda$
such that, for all $\alpha<\beta<\lambda^+$, $|Z_\alpha\cap Z_\beta|<\lambda$.
For all $\beta<\lambda^+$ and $\iota<\lambda$, let $Z_\beta(\iota)$ denote the unique $\zeta\in Z_\beta$ to satisfy $\otp(Z_\beta\cap\zeta)=\iota$.
For every $\beta<\lambda^+$, fix an injection $i_\beta:\beta\rightarrow\lambda$.
Then, define a partition $p:[\lambda^+]^2\rightarrow\lambda$ via:
$$p(\alpha,\beta):=\begin{cases}
g_{\beta}(\alpha)&\text{if }\alpha\in\dom(g_{\beta});\\
Z_\beta({\gamma_\beta}+i_\beta(\alpha))&\text{otherwise}.\end{cases}$$

A moment's reflection makes it clear that  $p$ has injective and $\chi$-Cohen fibers. 

To see that $p$ is $\lambda$-almost-disjoint,
fix an arbitrary pair $(\beta,\beta')\in[\lambda^+]^2$ and consider the set
$$A:=\{p(\alpha,\beta)\mid \alpha<\beta\}\cap\{p(\alpha,\beta')\mid \alpha<\beta\}.$$
Clearly, $|A|\le|g_{\beta}|+|g_{{\beta'}}|+|Z_\beta\cap Z_{\beta'}|<\lambda$, as sought.
\end{proof}

\begin{defn}\label{zp}
Let $p:[\kappa]^2\rightarrow\mu$ be a partition
and $c:[\kappa]^2\rightarrow\theta$ be a coloring.
\begin{itemize}
 \item  A subset $A\s\kappa$ is \emph{$(p,c)$-homogeneous} iff there exists a function
$\tau:\mu\rightarrow\theta$ such that  $c(\alpha,\beta)=\tau(p(\alpha,\beta))$
 for all $(\alpha,\beta)\in [A]^2$;
 \item   The \emph{$p$-cochromatic number of $c$}, denoted $z_p(c)$, is the least cardinal $\zeta$ such that $\kappa$ may be covered by $\zeta$ many $(p,c)$--homogeneous sets.
 \end{itemize}
\end{defn}

\section{Relations over partitions from forcing axioms}\label{forcingaxioms}
Our forcing convention is that $q\le p$ means that $q$ is a forcing condition
which extends the forcing condition $p$.
A notion of forcing $\mathbb Q$ has Knaster's Property (Property~K) iff for every uncountable set $A$ of conditions in $\mathbb Q$,
there is an uncountable $B\subseteq A$ such that any two elements of $B$ are compatible.

\begin{defn}\label{mak} $\MA_{\aleph_1}(K)$ asserts that for every notion of forcing $\mathbb Q$ having Property~K,
for every sequence $\langle D_\beta\mid \beta<\omega_1\rangle$ of dense subsets of $\mathbb Q$,
there is a filter $G$ over $\mathbb Q$ that meets each of the $D_\beta$'s.
\end{defn}

A notion of forcing $\mathbb Q=(Q,{\le})$ is \emph{well-met} iff every
pair $q_0,q_1$ of compatible conditions has a greatest lower bound,
i.e., an $r\le q_0,q_1$ such that for any condition $s$, if
$s\le q_0,q_1$ then $s\le r$.  The notion of forcing $\mathbb Q$ is said to
satisfy the \emph{$\lambda^+$-stationary chain condition}
($\lambda^+$-stationary-cc, for short) iff for every sequence $\langle q_\delta\mid \delta<\lambda^+\rangle$ of conditions in $\mathbb Q$
there is a club $D\s\lambda^+$ and a regressive map $h:D\cap E^{\lambda^+}_{\cf(\lambda)}\rightarrow\lambda^+$ such that for all
$\gamma,\delta\in\dom(h)$, if  $h(\gamma)=h(\delta)$ then $q_\gamma$
and $q_\delta$ are compatible.

\begin{defn}[Generalized Martin's Axiom]\label{defGMA}
  $\gma_{\lambda^+}$ asserts that for every notion of forcing $\mathbb Q=(Q,{\le})$ of size $<2^\lambda$ which satisfies the following conditions:
\begin{enumerate}
\item[(a)] $\mathbb Q$ is well-met;
\item[(b)] For all $\sigma<\lambda$, every $\le$-decreasing sequence
  of conditions $\langle q_i\mid i<\sigma\rangle$ in $\mathbb Q$  admits a greatest
  lower bound;
\item[(c)] $\mathbb Q$ satisfies the $\lambda^+$-stationary-cc,
\end{enumerate}
for every sequence $\langle D_\beta\mid \beta<\lambda^+\rangle$ of dense subsets of $\mathbb Q$
there is a filter $G$ over $\mathbb Q$ that meets each of the $D_\beta$'s.
\end{defn}

By Fodor's lemma, any poset satisfying the $\omega_1$-stationary-cc has Property~$K$.
Thus, $\MA_{\aleph_1}\implies\MA_{\aleph_1}(K)\implies\gma_{\aleph_1}$.

\begin{thm}[Shelah, \cite{Sh:80}] Suppose the  $\gch$ holds.  Then for any prescribed regular cardinal $\lambda$
there is a cofinality-preserving forcing extension in which
$\lambda^{<\lambda}=\lambda$ and $\gma_{\lambda^+}$ holds.
\end{thm}
\begin{remark}
The conjunction of $\lambda^{<\lambda}=\lambda$ and
$\gma_{\lambda^+}$ implies that $2^\lambda>\lambda^+$.  
Otherwise,  fix an enumeration
$\langle f_\beta\mid\beta<\lambda^+\rangle$ of ${}^\lambda\lambda$ and appeal to
$\gma_{\lambda^+}$ with $\mathbb Q:=\add(\lambda,1)$ and
$D_\beta:=\{q\in{}^{<\lambda}\lambda\mid q\nsubseteq f_\beta\}$ for each
$\beta<\lambda^+$.
\end{remark}

Together with Proposition~\ref{Wildkjsbdkj}(2),
the next result is Clause~(1) of Theorem~A.

\begin{thm}\label{thm51}  Suppose $\lambda^{<\lambda}=\lambda$ and $\gma_{\lambda^+}$ holds.
Let $p:[\lambda^+]^2\rightarrow\lambda$ be any partition with
injective and $\lambda$-almost-disjoint fibers.

For every coloring $c:[\lambda^+]^2\rightarrow\lambda$,
there exists a decomposition $\lambda^+=\biguplus_{i<\lambda}X_i$ such that for all $i<\lambda$:
\begin{itemize}
\item $X_i$ is $(p,c)$-homogeneous (recall Definition~\ref{zp});
\item if $p$ has in addition $\lambda$-Cohen fibers, then $|X_i|=\lambda^+$ and $p[[X_i]^2]=\lambda$.
\end{itemize}

In particular, $z_p(c)\le\lambda$ for every coloring $c:[\lambda^+]^2\rightarrow\lambda$.
\end{thm}
\begin{proof} Fix an arbitrary coloring $c:[\lambda^+]^2\rightarrow\lambda$.
Define a notion of forcing $\mathbb Q=(Q,\supseteq)$, where $Q$
consists of all functions $f:a\rightarrow\lambda$ such that:
\begin{enumerate}
\item $a\in[\lambda^+]^{<\lambda}$;
\item \label{Cauhku64h} for all $i,j<\lambda$, the set
$$\Gamma_{i,j}^f:=\{ c(\alpha,\beta)\mid (\alpha,\beta)\in[a]^2,\ f(\alpha)=i=f(\beta),\ p(\alpha,\beta)=j\}$$
contains at most one element.
\end{enumerate}
It will be shown that $\mathbb Q$ satisfies the requirement of Definition~\ref{defGMA}.
We first dispose of an easy claim.

\begin{claim}\label{claim811} Let $\mathcal F$ be a centered family of conditions of size $<\lambda$. Then $\bigcup\mathcal F$ is a condition.
\end{claim}
\begin{proof} Suppose not. Denote $f:=\bigcup\mathcal F$ and $a:=\dom(f)$.
As $a\in[\lambda^+]^{<\lambda}$ this must mean that there are  $i,j<\lambda$ for which the set $\Gamma_{i,j}^f$ has more than one element.
This means that we may pick $(\alpha_0,\beta_0),(\alpha_1,\beta_1)\in[a]^2$ such that:
\begin{itemize}
\item $f(\alpha_0)=f(\alpha_1)=i=f(\beta_0)=f(\beta_1)$;
\item $p(\alpha_0,\beta_0)=j=p(\alpha_1,\beta_1)$;
\item  $c(\alpha_0,\beta_0)\neq c(\alpha_1,\beta_1)$.
\end{itemize}
Fix $f_0,f_1,f^0,f^1\in\mathcal F$ such that $\alpha_0\in\dom(f_0)$, $\alpha_1\in\dom(f_1)$, $\beta_0\in\dom(f^0)$ and $\beta_1\in\dom(f^1)$.
Since $\mathcal F$ is centered, there exists a condition $q$ such that $f_0\cup f_1\cup f^0\cup f^1\s q$.
A moment's reflection makes it clear that $q(\alpha_0)=q(\alpha_1)=i=q(\beta_0)=q(\beta_1)$,
so since $p(\alpha_0,\beta_0)=j=p(\alpha_1,\beta_1)$ and since $q$ is a legitimate condition,
this must mean that $c(\alpha_0,\beta_0)=c(\alpha_1,\beta_1)$. This is a contradiction.
\end{proof}

\begin{claim} For every $\beta<\lambda^+$, $D_\beta:=\{ f\in Q\mid \beta\in\dom(f)\}$ is dense in $\mathbb Q$.
\end{claim}
\begin{proof} Let $\beta<\lambda^+$. Given any condition $f$ such that $\beta\notin\dom(f)$,
we get that $f':=f\cup\{(\beta,\sup(\im(f))+1)\}$ is an extension of $f$ in $D_\beta$.
\end{proof}

It thus follows that if $G$ is a filter over $\mathbb Q$ that meets each of the $D_\beta$'s,
then $g:=\bigcup G$ is a function from $\lambda^+$ to $\lambda$ such that
for every $i<\lambda$,
if we let $X_i:=\{ \alpha<\lambda^+\mid g(\alpha)=i\}$, then
$|\{ c(\alpha,\beta)\mid (\alpha,\beta)\in [X_i]^2\ \&\ p(\alpha,\beta)=j\}|\le 1$ 
for all $j<\lambda$. Thus, $z_p(c)\le \lambda$ holds in the
forcing extension by $\mathbb Q$. 

\begin{claim} If $p$ has $\lambda$-Cohen fibers, then, for all $\epsilon<\lambda^+$
and $i,j<\lambda$:
\begin{enumerate}
\item  
$D^\epsilon:=\{ f\in Q\mid \exists \beta\in\dom(f)\,(\beta\ge\epsilon\ \&\ f(\beta)=i)\}$ is dense;
\item 
 $D_{i,j}:=\{f\in Q\mid \exists(\alpha,\beta)\in[\dom(f)]^2, f(\alpha)=i=f(\beta)\ \&\ p(\alpha,\beta)=j\}$ is dense.
\end{enumerate}
\end{claim}
\begin{proof} Suppose that $p$ has $\lambda$-Cohen fibers.

(1) Given $\epsilon<\lambda^+$ and a condition $f:a\rightarrow\lambda$,
fix a large enough $\eta<\lambda^+$ such that $\{ p(\alpha,\beta)\mid (\alpha,\beta)\in[a]^2\}\s\eta$, 
and then define an injection $g:a\rightarrow\lambda$ via
$$g(\alpha):=\eta+\otp(a\cap \alpha).$$
Now, as $p$ has $\lambda$-Cohen fibers, we may find some $\beta<\lambda^+$ with $a\cup \epsilon\s\beta$
such that $p(\alpha,\beta)=g(\alpha)$ for all $\alpha\in a$.
It is clear that $f':=f\cup\{(\beta,i)\}$ is an extension of $f$ lying in $D^\epsilon$.

(2) Given $i,j<\lambda^+$ and a condition $f:a\rightarrow\lambda$, we do the following.
First, by Clause~(1), we may assume the existence of some $\alpha^*\in a$ such that $f(\alpha^*)=i$.
Of course, if $f\in D_{i,j}$, then we are done,
thus, hereafter, assume that $f\notin D_{i,j}$.

Fix a large enough $\eta\in\lambda^+\setminus(j+1)$ such that $\{ p(\alpha,\beta)\mid (\alpha,\beta)\in[a]^2\}\s\eta$.
Define an injection $g:a\rightarrow\lambda$ via
$$g(\alpha):=\begin{cases}
j&\text{if }\alpha=\alpha^*;\\
\eta+\otp(a\cap \alpha)&\text{otherwise}.
\end{cases}$$

Now, as $p$ has $\lambda$-Cohen fibers, we may find some  $\beta<\lambda^+$ with $a\s\beta$
such that $p(\alpha,\beta)=g(\alpha)$ for all $\alpha\in a$.
As $f\not\in D_{i,j}$, it immediately follows that 
$f':=f\cup\{(\beta,i)\}$ is a legitimate condition lying in $D_{i,j}$. So we are done.
\end{proof}

Clearly, $\mathbb Q$ has size no more than
$|[\lambda^+]^{<\lambda}|=\lambda^+<2^\lambda$. In addition, by
Claim~\ref{claim811}, Clauses (a) and (b) of Definition~\ref{defGMA}
hold true.  Thus, to complete the proof, we are left with addressing
Clause (c) of Definition~\ref{defGMA}.  To this end, assume we are
given a sequence $\langle f_\delta\mid \delta<\lambda^+\rangle$ of
conditions in $\mathbb Q$; we need to find a club $D\s\lambda^+$ and a
regressive map $h:D\cap E^{\lambda^+}_\lambda\rightarrow\lambda^+$
such that for all $\gamma,\delta\in D\cap E^{\lambda^+}_\lambda$, if 
$h(\gamma)=h(\delta)$ then  $f_\gamma$ and  $f_\delta$ are compatible.

Consider the club $C:=\{\delta<\lambda^+\mid \forall\gamma<\delta\,(\sup(\dom(f_\gamma))<\delta)\}$.
Fix an injective enumeration $\langle (\psi_\tau,\varphi_\tau,\xi_\tau,\mu_\tau,\epsilon_\tau)\mid \tau<\lambda^+\rangle$ of
$[\lambda^3]^{<\lambda}\times[\lambda^+\times\lambda]^{<\lambda}\times \lambda\times\lambda\times \lambda^+$,
and then consider the subclub:
$$D:=\{\delta\in C\mid  \{(\psi_\tau,\varphi_\tau,\xi_\tau,\mu_\tau,\epsilon_\tau)\mid \tau<\delta\}=[\lambda^3]^{<\lambda}\times[\delta\times\lambda]^{<\lambda}\times\lambda\times\lambda\times \delta\}.$$

We define the function $h:D\cap E^{\lambda^+}_\lambda\rightarrow\lambda^+$ as follows.
Given $\delta\in D\cap E^{\lambda^+}_\lambda$, let $h(\delta):=\tau$
for the least $\tau<\delta$
which satisfies all of the following:
\begin{itemize}
\item[(a)] $\left\{ (i,j,c(\alpha,\beta))\mid i,j<\lambda,(\alpha,\beta)\in[\dom(f_\delta)]^2, f_\delta(\alpha)=i=f_\delta(\beta), p(\alpha,\beta)=j\right\}=\psi_\tau$;
\item[(b)] $f_\delta\restriction \delta=\varphi_\tau$;
\item[(c)] $\bigcup\{\{p(\alpha,\beta)\mid \alpha<\beta\}\cap\{p(\alpha,\beta')\mid \alpha<\beta\}\mid {(\beta,\beta')\in[\dom(f_\delta)]^2}\}\s\xi_\tau$;
\item[(d)] $\{ p(\alpha,\beta)\mid (\alpha,\beta)\in[\dom(f_\delta)]^2\}\s\mu_\tau$;
\item[(e)] $\{ \alpha<\delta\mid \exists \beta\in\dom(f_\delta)\setminus\delta~[p(\alpha,\beta)\le\max\{\xi_\tau,\mu_\tau\}]\}\s\epsilon_\tau$.
\end{itemize}

\begin{claim} $h$ is well-defined.
\end{claim}
\begin{proof} Let $\delta\in D\cap E^{\lambda^+}_\lambda$. First we make note of the following:
\begin{itemize}
\item The corresponding set of Clause~(a) is a subset of $\lambda^3$ of size $<\lambda$;
\item The corresponding set of Clause~(b) is a subset of $\delta\times\lambda$ of size $<\lambda$;
\item As $|\dom(f_\delta)|<\lambda$, the fact that $p$ has $\lambda$-almost-disjoint fibers ensures that an ordinal $\xi_\tau<\lambda$ as in Clause~(c) exists;
\item As $|\dom(f_\delta)|<\lambda$, an ordinal $\mu_\tau<\lambda$ as in Clause~(d) does exist;
\item As $|\dom(f_\delta)|<\lambda=\cf(\delta)$, the fact that $p$ has
injective fibers ensures that an ordinal $\epsilon_\tau<\lambda$ as in Clause~(e) exists.
\end{itemize}
So, since $\delta\in D$, a $\tau<\delta$ for which Clause (a)--(e) are satisfied does exist.
\end{proof}
To see that $h$ is as sought, fix a pair $\gamma<\delta$ of ordinals in $D\cap E^{\lambda^+}_\lambda$ such that $h(\gamma)=h(\delta)$,
say, both are  equal to $\tau$.
As $\delta\in C$, Clause~(b) implies that $f:=f_\gamma\cup f_\delta$ is a function.
To see that $f\in Q$, it suffices to verify Clause~(2) above
with $a:=\dom(f)$.
Towards a contradiction, suppose that there $i,j<\lambda$ and $(\alpha_0,\beta_0),(\alpha_1,\beta_1)\in[a]^2$ such that:
\begin{itemize}
\item $f(\alpha_0)=f(\alpha_1)=i=f(\beta_0)=f(\beta_1)$;
\item $p(\alpha_0,\beta_0)=j=p(\alpha_1,\beta_1)$;
\item  $c(\alpha_0,\beta_0)\neq c(\alpha_1,\beta_1)$.
\end{itemize}

Denote $a_\gamma:=\dom(f_\gamma)$ and $a_\delta:=\dom(f_\delta)$.
As $h(\gamma)=\tau=h(\delta)$, Clause~(a) implies that it cannot be the case that $\{(\alpha_0,\beta_0),(\alpha_1,\beta_1)\}\s [a_\gamma]^2\cup[a_\delta]^2$.
So, without loss of generality, assume that $(\alpha_0,\beta_0)\notin [a_\gamma]^2\cup[a_\delta]^2$.
By Clause~(b), in particular, $a_\gamma\cap\gamma=a_\delta\cap\delta$,
and so, since $(\alpha_0,\beta_0)\notin [a_\gamma]^2\cup[a_\delta]^2$, it must be the case that $\alpha_0\ge\gamma$ and $\beta_0\ge\delta$.
If $\alpha_0\ge\delta$, then since $\delta\in C$, we would get that $(\alpha_0,\beta_0)\in[a_\delta]^2$,
which is not the case. Altogether, $\gamma\le\alpha_0<\delta\le\beta_0$.
In particular, $\alpha_0\in(a_\gamma\setminus\gamma)$ and $\beta_0\in(a_\delta\setminus\delta)$.

By Clause~(e), $\epsilon_\tau<\gamma$, and hence $\alpha_0>\epsilon_\tau$.
It thus follows from Clause~(e) that $p(\alpha_0,\beta_0)>\max\{\xi_\tau,\mu_\tau\}$.
So $p(\alpha_1,\beta_1)=j=p(\alpha_0,\beta_0)>\mu_\tau$. Recalling Clause~(d), this means that $(\alpha_1,\beta_1)\notin [a_\gamma]^2\cup[a_\delta]^2$.
Hence, the same analysis we had for $(\alpha_0,\beta_0)$ is valid also for $(\alpha_1,\beta_1)$.
In particular, $\gamma\le\alpha_1<\delta\le\beta_1$ (so that $\{\beta_0,\beta_1\}\s a_\delta\setminus\delta$)
and $p(\alpha_1,\beta)>\xi_\tau$.
By Clause~(c) for $\gamma$, $\xi_\tau<\gamma<\alpha_0,\alpha_1$
and then by Clause~(c) for $\delta$ we infer that if $\beta_0\neq\beta_1$, then $p(\alpha_0,\beta_0)\neq p(\alpha_1,\beta_1)$.
It thus follows that $\beta_0=\beta_1$. As $p$ has has injective fibers it  follows that $\alpha_0=\alpha_1$,
contradicting the fact that $c(\alpha_0,\beta_0)\neq c(\alpha_1,\beta_1)$.
\end{proof}

The next result answers two questions from \cite{strongcoloringpaper}:
Question~48 in the negative and Question~49 in the affirmative. Recall that by
Proposition~\ref{Wildkjsbdkj}(2), the set of partitions $p:[\omega_1]^2\rightarrow\omega$
with injective, $\aleph_{1}$-almost-disjoint and $\aleph_{0}$-Cohen fibers is not vacuous.

\begin{cor}\label{threeitems} Suppose $\MA_{\aleph_1}(K)$ holds.  Then for every partition
  $p:[\omega_1]^2\rightarrow\omega$   with injective,
  $\aleph_{1}$-almost-disjoint and $\aleph_{0}$-Cohen fibers, for every coloring
  $c:[\omega_1]^2\rightarrow\omega$ there exists a decomposition $\aleph_1=\biguplus_{i<\omega}X_i$ such that,
  for every $i<\omega$:
  \begin{itemize}
  \item $X_i$ is uncountable;
  \item $X_i$ is $p$-omnichromatic, i.e., $p[[X_i]^2]=\omega$;
   \item $X_i$ is $(p,c)$-homogeneous, i.e., for every $j<\omega$, 
   $$c\restriction\{ (\alpha,\beta)\in[ X_i]^2\mid p(\alpha,\beta)=j\}\text{ is constant}.$$
  \end{itemize}
\end{cor}
\begin{proof} By Theorem~\ref{thm51}.
\end{proof}

It may be interesting to point out that the middle item in Corollary \ref{threeitems} is not required for establishing $z_{p}(c)\le \aleph_{0}$, but nevertheless holds.
It can be interpreted as a silent witness for $p$ being, from the viewpoint of a coloring, anti-Ramsey.

The upcoming theorem applies to a broader set of partitions than the one in \ref{threeitems} in return for
 allowing  a finite number of colors rather than a single color.
It also implies Clause~(2) of Theorem~A. 

\begin{thm}\label{thm56}
Suppose $\MA_{\aleph_1}(K)$ holds. Then for every partition $p$ which witnesses $\U(\omega_1,\allowbreak\omega_1,\omega,\omega)$,
and every coloring $c:[\omega_1]^2\rightarrow\omega$,
 there is a decomposition $\omega_1=\biguplus_{i<\omega}X_i$
such that for all $i,j<\omega$, $$\{ c(\alpha,\beta)\mid (\alpha,\beta)\in[ X_i]^2\ \&\ p(\alpha,\beta)=j\}\text{ is finite}.$$
\end{thm}
\begin{proof} Define a notion of forcing $\mathbb Q$ consisting of conditions $q=(m_q,f_q,g_q)$ as follows:
\begin{enumerate}
\item $m_q<\omega$;
\item $f_q:m_q\times m_q\rightarrow\omega$ is a function;
\item $g_q:a_q\rightarrow m_q$ is a function, with $a_q\in[\omega_1]^{<\aleph_0}$;
\item for all $(\alpha,\beta)\in[a_q]^2$, $p(\alpha,\beta)<m_q$;
\item for all $(\alpha,\beta)\in[a_q]^2$, if $g_q(\alpha)=g_q(\beta)$, then  $c(\alpha,\beta)<f_q(g_q(\alpha),p(\alpha,\beta))$.
\end{enumerate}

A condition $q$ extends a condition $\bar q$ iff $m_q\ge m_{\bar q}$,
$f_q\supseteq f_{\bar q}$ and $g_q\supseteq g_{\bar q}$.

\begin{claim} For every $\beta<\omega_1$, $D_\beta:=\{ q\mid \beta\in a_q\}$ is dense in $\mathbb Q$
\end{claim}
\begin{proof} Let $\beta<\omega_1$. Given any condition $(m,f,g)$ in $\mathbb Q$ such that $\beta\notin\dom(g)$,
let $$m':=\max\{m,p(\alpha,\beta)\mid \alpha\in\dom(g)\}$$
and define a condition $q=(m_q,f_q,g_q)$,
by letting $m_q:=m'+1$,
letting $f_q:m_q\times m_q\rightarrow\omega$ be an arbitrary function extending $f$,
and letting $g_q:=g\cup\{(\beta,m')\}$.
\end{proof}

It thus follows that if $G$ is a filter over $\mathbb Q$ such that $G\cap D_\beta\neq\emptyset$ for all $\beta<\omega_1$,
then by letting $f:=\bigcup\{ f_q\mid q\in G\}$ and $g:=\bigcup\{ g_q\mid q\in G\}$,
we get that $\dom(f)=\omega\times\omega$, $\dom(g)=\omega_1$,
and for every $i,j<\omega$, if we let $X_i:=\{\beta<\omega_1\mid g(\beta)=i\}$,
then $\{ c(\alpha,\beta)\mid (\alpha,\beta)\in[ X_i]^2\ \&\  p(\alpha,\beta)=j\}\s f(i,j)$.

Now, let us verify that $\mathbb Q$ has Property~$K$.
To this end, suppose that $A$ is an uncountable family of conditions in $\mathbb Q$. By the pigeonhole principle, 
we may assume the existence of an integer $m$ and a function $f$
such that, for all $q\in A$, $m_q=m$ and $f_q=f$.
By the $\Delta$-system lemma, it may also be assumed that $\{ a_q\mid q\in A\}$ forms a $\Delta$-system with some root $r$.
For any $q\in A$, denote $a_q':=a_q\setminus r$.
Since the $a_q$'s are subsets of $\omega_1$, it can be assumed
that $r$ forms an initial segment of $a_q$ for each $q$ and that, if $q\neq \bar{q}$, then either $\max(a'_q)<\min(a'_{\bar{q}})$ or $\max(a'_{\bar q})<\min(a'_{q})$.
By shrinking further, we may assume that $q\mapsto g_q\restriction r$ is constant over $A$.
Next, by the choice of the partition $p$, we fix an uncountable $B\s A$ with the property that for all $\bar q,q\in B$,
if $\max(a_{\bar q}')<\min(a_q')$, then $\min(p[a_{\bar q}'\times a_q'])> m$.

To see that $\{ a_q\mid q\in B\}$ is directed, fix two conditions $\bar q\neq q$ in $B$.
Without loss of generality, we may assume that $\max(a_{\bar q}')<\min(a_q')$.

Set $a^*:=a_{\bar q}\cup a_q$,
$g^*:=g_{\bar q}\cup g_q$ and $m^*:=\max(p[a^*]^2)+1$.

\begin{claim}\label{claim722} For every $(\alpha,\beta)\in[a^*]^2\setminus([a_{\bar q}]^2\cup[a_{ q}]^2)$,
$(\alpha,\beta)\in a_{\bar q}'\times a_q'$.
\end{claim}
\begin{proof} Let $(\alpha,\beta)\in[a^*]^2\setminus([a_{\bar q}]^2\cup[a_{ q}]^2)$.
As $r=a_{\bar q}\cap a_q$, we infer that $\{\alpha,\beta\}\cap r=\emptyset$.
So $\{\alpha,\beta\}\cap a_{\bar q}'$ and $\{\alpha,\beta\}\cap a_{q}'$ are singletons.
Since $\alpha<\beta$ and $\max(a_{\bar q}')<\min(a_q')$, it altogether follows that $(\alpha,\beta)\in a_{\bar q}'\times a_q'$.
\end{proof}

Fix any function $f^*:m^*\times m^*\rightarrow\omega$ extending $f$ by letting, for all $(i,j)\in(m^*\times m^*)\setminus(m\times m)$,
$$f^*(i,j):=\max\{ c(\alpha,\beta)\mid (\alpha,\beta)\in[a^*]^2, g(\alpha)=i=g(\beta)\ \&\ p(\alpha,\beta)=j)\}+1.$$

Looking at Clauses (1)--(5) above, it is clear that for
$q^*:=(m^*,f^*,g^*)$ to be a condition in $\mathbb Q$ it suffices to
verify the following claim.

\begin{claim} Let $(\alpha,\beta)\in[a^*]^2$ with $g^*(\alpha)=g^*(\beta)$.
Then $$c(\alpha,\beta)<f^*(g^*(\alpha),p(\alpha,\beta)).$$
\end{claim}
\begin{proof} Denote $i:=g^*(\alpha)$ and $j:=p(\alpha,\beta)$.
We shall show that $c(\alpha,\beta)<f^*(i,j)$.

Of course, if $(\alpha,\beta)\in[a_{\bar q}]^2$, then $g_{\bar q}(\alpha)=i<m$, $p(\alpha,\beta)<m$
and $c(\alpha,\beta)=f_{\bar q}(i,j)=f^*(i,j)$.
Likewise, if $(\alpha,\beta)\in[a_{ q}]^2$, then  $c(\alpha,\beta)=f_{q}(i,j)=f^*(i,j)$.

Next, assume that $(\alpha,\beta)\notin[a_{\bar q}]^2\cup [a_q]^2$.
So, by Claim~\ref{claim722}, $(\alpha,\beta)\in a_{\bar q}'\times a_q'$.
As $j\ge\min(p[a_{\bar q}'\times a_q'])> m>\max(p[[a_{\bar q}]^2\cup[a_{ q}]^2])$,
we infer that $(i,j)\in(m^*\times m^*)\setminus(m\times m)$ and hence the definition of $f^*(i,j)$ makes it clear that
$c(\alpha,\beta)<f^*(i,j)$, as sought.
\end{proof}
So $q^*$ is a legitimate condition witnessing that $\bar q$ and $q$ are compatible.
Thus, we have demonstrated that $\mathbb Q$ indeed satisfies Property~$K$.
\end{proof}

We now present two $\zfc$ results which show  that the preceding is optimal.
To see how the first result connects to Theorem~\ref{thm56} note that any partition $p:[\omega_1]^2\rightarrow\omega$ with injective (or just finite-to-one) fibers
witnesses $\U(\omega_1,\omega_1,\omega,\omega)$.\footnote{For every regular uncountable cardinal $\kappa$ that is not greatly Mahlo, there is a partition $p:[\kappa]^2\rightarrow\omega$ with nowhere bounded-to-one fibers that nevertheless satisfy $\U(\kappa,\kappa,\omega,\omega)$; see \cite[Lemma~2.12(3), Lemma~2.2(3) and the proof Lemma~5.8]{paper35}.}

\begin{thm} There exist a partition $p:[\omega_1]^2\rightarrow\omega$ with injective fibers and a coloring $c:[\omega_1]^2\rightarrow\omega$
such that, for every $k<\omega$, and every $X\s\omega_1$ with $\otp(X)=\omega+k$, there exists $j<\omega$
such that $|\{c(\alpha,\beta)\mid (\alpha,\beta)\in[X]^2\ \&\ p(\alpha,\beta)=j\}|\ge k$.
\end{thm}
\begin{proof}  By Proposition~\ref{Wildkjsbdkj}, let us fix a partition $p:[\omega_1]^2\rightarrow\omega$
with injective and $\omega$-coherent fibers, and a coloring $c:[\omega_1]^2\rightarrow\omega$ with injective and $\omega$-almost-disjoint fibers.
Now, given $k<\omega$ and an increasing sequence $\langle \xi_n\mid n<\omega+k\rangle$ of countable ordinals, we do the following.
For each $i<k$, denote $\beta_i:=\xi_{\omega+i}$.
\begin{itemize}
\item As $p$ has $\omega$-coherent fibers, $a:=\bigcup_{i<i'<k}\{ \alpha<\beta_0\mid p(\alpha,\beta_i)\neq p(\alpha,\beta_{i'})\}$ is finite;
\item As $c$ has $\omega$-almost-disjoint fibers,
$T:=\bigcup_{i<i'<k}\{c(\alpha,\beta_i)\mid \alpha<\beta_0\}\cap\{ c(\alpha,\beta_{i'})\mid \alpha<\beta_0\}$ is finite;
\item As $c$ has injective fibers, $a':=\bigcup_{i<k}\{\alpha<\beta_0\mid c(\alpha,\beta_i)\in T\}$ is finite.
\end{itemize}

Now, pick $n<\omega$ such that $\xi_n\notin a\cup a'$, and set $\alpha:=\xi_n$. Let $j:=p(\alpha,\beta_0)$.
\begin{itemize}
\item As $\alpha\in \beta_0\setminus a$, we infer that $p(\alpha,\beta_i)=j$ for all $i<k$;
\item As $\alpha\in \beta_0\setminus a'$, we infer that $c(\alpha,\beta_i)\neq c(\alpha,\beta_{i'})$ for all $i<i'<k$.
\end{itemize}
So $|\{c(\alpha,\beta)\mid (\alpha,\beta)\in[X]^2\ \&\ p(\alpha,\beta)=j\}|\ge k$.
\end{proof}

\begin{thm}\label{thm49}
For every partition $p:[\omega_1]^2\rightarrow\omega$ and every uncountable $X\s\omega_1$ such that $p\restriction[X]^2$ does not witness $\U(\omega_1,\allowbreak\omega_1,\omega,\omega)$,
for every coloring $c:[\omega_1]^2\rightarrow\omega$ with finite-to-one fibers,
there exists $j<\omega$ such that $$\{ c(\alpha,\beta) \mid (\alpha,\beta)\in[X]^2, p(\alpha,\beta)=j\}\text{ is infinite}.$$
\end{thm}
\begin{proof} Suppose  $p$ and $X$ are as above.
Fix $n,k<\omega$ and an uncountable pairwise disjoint family $\mathcal A\s[X]^k$,
such that for every uncountable $\mathcal B\s\mathcal A$ there is a pair $(a,b)\in[\mathcal B]^2$
such that $p[a\times b]\cap n\neq\emptyset$. By the Dushhnik-Miller theorem, then,
there exists a $<$-increasing sequence  $\langle a_i \mid i<\omega+1\rangle$ of elements of $\mathcal A$
such that $p[a_i\times a_{i'}]\cap n\neq\emptyset$ for all $i<i'<\omega+1$.
It follows that there exist $I\in[\omega]^\omega$, $\beta\in a_\omega$,  and $\langle \alpha_i\mid i\in I\rangle\in\prod_{i\in I}a_i$
such that $i\mapsto p(\alpha_i,\beta)$ is constant over $I$ with some value $j<n$.
Then, for every coloring $c:[\omega_1]^2\rightarrow\omega$ with finite-to-one fibers,
the set $\{ c(\alpha,\beta) \mid (\alpha,\beta)\in[X]^2, p(\alpha,\beta)=j\}$ is infinite.
\end{proof}

For a partition $p:[\kappa]^2\rightarrow\mu$, denote by $\kappa\rightarrow_p[\kappa]^2_{\theta,{<}\theta'}$
the assertion that for every coloring $c:[\kappa]^2\rightarrow\theta$,
there is $X\s\kappa$ of size $\kappa$ such that, for any cell $j<\mu$,
$$|\{ c(\alpha,\beta)\mid (\alpha,\beta)\in[ X]^2\ \&\ p(\alpha,\beta)=j\}|<\theta'.$$

The next result is Theorem~B.
\begin{cor} Assume $\MA_{\aleph_1}(K)$. Then for every partition $p:[\omega_1]^2\rightarrow\omega$, the following are equivalent:
\begin{enumerate}
\item $\omega_1\rightarrow_p[\omega_1]^2_{\omega,\text{finite}}$;
\item There exists $X\in[\omega_1]^{\aleph_1}$ such that $p\restriction[X]^2$ witnesses $\U(\omega_1,\allowbreak\omega_1,\omega,\omega)$.
\end{enumerate}
\end{cor}
\begin{proof} 
$(1)\implies(2)$: Fix any coloring $c:[\omega_1]^2\rightarrow\omega$ with finite-to-one fibers.
Assuming that $\omega_1\rightarrow_p[\omega_1]^2_{\omega,\text{finite}}$ holds,
let us now fix $X\in[\omega_1]^{\aleph_1}$ that witnesses the instance $\omega_1\rightarrow_p[\omega_1]^2_{\omega,\text{finite}}$ for the coloring $c$.
This means that 
$\{ c(\alpha,\beta)\mid (\alpha,\beta)\in[X]^2\ \&\ p(\alpha,\beta)=j\}$ is finite for every $j<\omega$.
So, by Theorem~\ref{thm49},  $p\restriction[X]^2$ must witness $\U(\omega_1,\omega_1,\omega,\omega)$.

$(2)\implies(1)$: 
Fix $X\in[\omega_1]^{\aleph_1}$ such that $p\restriction[X]^2$ witnesses $\U(\omega_1,\allowbreak\omega_1,\omega,\omega)$.
Then, by Theorem~\ref{thm56}, for every coloring $c:[\omega_1]^2\rightarrow\omega$, 
 there is a decomposition $X=\biguplus_{i<\omega}X_i$
such that, for all $i,j<\omega$, $\{ c(\alpha,\beta)\mid (\alpha,\beta)\in[ X_i]^2\ \&\ p(\alpha,\beta)=j\}$ is finite.
Fix $i<\omega$ such that $X_i$ is uncountable.
Then $X_i$ witnesses the instance $\omega_1\rightarrow_p[\omega_1]^2_{\omega,\text{finite}}$ for the coloring $c$.
\end{proof}

The same proof yields:
\begin{cor} Assuming $\MA_{\aleph_1}(K)$, for every partition $p:[\omega_1]^2\rightarrow\omega$, the following are equivalent:
\begin{enumerate}
\item There is a decomposition $\omega_1=\biguplus_{i<\omega}X_i$
such that, for all $i,j<\omega$, $$\{ c(\alpha,\beta)\mid (\alpha,\beta)\in[ X_i]^2\ \&\ p(\alpha,\beta)=j\}\text{ is finite};$$
\item There is a decomposition $\omega_1=\biguplus_{i<\omega}X_i$
such that, for all $i<\omega$,
$p\restriction[X_i]^2$ witnesses $\U(\omega_1,\allowbreak\omega_1,\omega,\omega)$.\qed
\end{enumerate}
\end{cor}

For completeness, we mention that by \cite[Corollary~29]{strongcoloringpaper}
it is consistent with $\MA_{\omega_1}(\sigma\text{-linked})$ that $\omega_1\nrightarrow_p[\omega_1]^2_{\omega}$
(in fact, $\Pr_{0}(\omega_1,\omega_1,\omega_1,\omega)_{p}$) holds
for any partition $p:[\omega_1]^2\rightarrow\omega$.

\section{Acknowledgments}

Kojman was partially supported by the Israel Science Foundation (grant agreement 665/20).
Rinot was partially supported by the Israel Science Foundation (grant agreement 2066/18)
and by the European Research Council (grant agreement ERC-2018-StG 802756).
Stepr\={a}ns was partially supported by NSERC of Canada.

\end{document}